\documentclass[12pt,a4paper,oneside]{amsart}

\usepackage[top=3cm, bottom=3cm, left=3cm, right=3cm, headheight=0.5cm]{geometry}

\usepackage{amssymb}
\usepackage{amsthm}

\usepackage[english]{babel}
\usepackage{bm}
\usepackage{tabu}
\usepackage{tabularx}
 
 \newcolumntype{Y}{>{\centering\arraybackslash}X}
\usepackage[shortlabels]{enumitem} 
  
\usepackage{accents}
\renewcommand*{\dot}[1]{%
  \accentset{\mbox{\bfseries .}}{#1}}

\newcommand{\rr}{\mathbb{R}}
\newcommand\dotprod[1]{\left<{#1}\right>}
\newtheorem{thm}{Theorem}

\newtheorem{prop}{Proposition}
\newtheorem{lemm}{Lemma}
\newtheorem{cor}{Corollary}
\theoremstyle{definition}
\newtheorem{defi}{Definition}
\newtheorem{rem}{Remark}

\title[Finsler metrics and semi-symmetric compatible linear connections]{Finsler metrics and semi-symmetric compatible linear connections}
\author{Csaba Vincze}
\address{Institute of Mathematics, University of Debrecen, H-4002 Debrecen, P.O.Box 400, Hungary}
\email{csvincze@science.unideb.hu}
\author{Márk Oláh}
\address{Institute of Mathematics, University of Debrecen, Doctoral School of Mathematical and Computational Sciences, H-4002 Debrecen, P.O.Box 400, Hungary}
\email{olah.mark@science.unideb.com}


\begin{document}

\maketitle

\begin{abstract} Finsler metrics are direct generalizations of Riemannian metrics such that the quadratic Riemannian indicatrices in the tangent spaces of a manifold are replaced by more general convex bodies as unit spheres. A linear connection on the base manifold is called compatible with the Finsler metric if the induced parallel transports preserve the Finslerian length of tangent vectors. Finsler manifolds admitting compatible linear connections are called generalized Berwald manifolds \cite{Wag1}. Compatible linear connections are the solutions of the so-called compatibility equations containing the components of the torsion tensor as unknown quantities. Although there are some theoretical results for the solvability of the compatibility equations  (monochromatic Finsler metrics \cite{BM}, extremal compatible linear connections and algorithmic solutions \cite{V14}), it is very hard to solve in general because compatible linear connections may or may not exist on a Finsler manifold and may or may not be unique.  Therefore special cases are of special interest. One of them is the case of the so-called semi-symmetric compatible linear connection with decomposable torsion tensor. It is proved \cite{V10} (see also \cite{V11}) that such a compatible linear connection must be uniquely determined. 
The original proof is based on averaging in the sense that the 1-form in the decomposition of the torsion tensor can be expressed by integrating differential forms on the tangent manifold over the Finslerian indicatrices. The integral formulas are very difficult to compute in practice. In what follows we present a new proof for the unicity result by using linear algebra and some basic facts about convex bodies. We present an explicit formula for the solution without integration. The method has a new contribution to the problem as well:  necessary conditions of the solvability are formulated in terms of intrinsic equations without unknown quantities. They are sufficient if and only if the solution depends only on the position.  
\end{abstract}

\section{Compatibility equations in Finsler geometry}

Let $M$ be a smooth connected manifold with a local coordinate system $u^1, \ldots, u^n$. The induced local coordinate system on the tangent manifold $TM$ consists of the functions $x^1, \ldots, x^n$ and $y^1, \ldots, y^n$ given by $x^i(v):=u^i\circ \pi (v)=u^i(p)=:p^i$, where  $\pi \colon TM \to M$ is the canonical projection and $y^i(v)=v(u^i)$, $i=1, \ldots, n$. Throughout the paper, we will use the shorthand notations
\begin{equation}
\partial_i := \dfrac{\partial}{\partial x^i} \hspace{1cm} \text{and} \hspace{1cm} \dot{\partial}_i:= \dfrac{\partial}{\partial y^i}.
\end{equation}

A Finsler metric \cite{BSC} on a manifold is a smoothly varying family of  Minkowski norms in the tangent spaces. It is a direct generalization of Riemannian metrics, with inner products (quadratic indicatrices) in the tangent spaces replaced by Minkowski norms (smooth strictly convex bodies).

\begin{defi} A \textbf{Finsler metric} is a non-negative continuous function $F\colon TM\to \mathbb{R}$ satisfying the following conditions:  $\displaystyle{F}$ is smooth on the complement of the zero section (\emph{regularity}), $\displaystyle{F(tv)=tF(v)}$ for all $\displaystyle{t> 0}$ (\emph{positive homogeneity}), $F(v)= 0$ if and only if $v={\bf 0}$  (\emph{definiteness}) and the Hessian $\displaystyle{g_{ij}=\dot{\partial}_{i}\dot{\partial}_{j}E}$ of the energy function $E=F^2/2$ is positive definite at all non-zero elements $\displaystyle{v\in T_pM}$ (\emph{strong convexity}). The pair $(M,F)$ is called a \textbf{Finsler manifold}.
\end{defi}

On a Riemannian manifold we obviously have compatible linear connections in the sense that the induced parallel transports preserve the Riemannian length of tangent vectors (metric linear connections). Following the classical Christoffel process it is clear that such a linear connection is uniquely determined by the torsion tensor. In contrast to the Riemannian case, non-Riemannian Finsler manifolds admitting compatible linear connections form a special class of spaces in Finsler geometry. They are called generalized Berwald manifolds \cite{Wag1}. It is known that some Finsler manifolds do not admit any compatible linear connections because of topological constraints, some have infinitely many compatible linear connections and it can also happen that the compatible linear connection is uniquely determined \cite{VOM}, see also  \cite{RandersGBM}.

\begin{defi} A linear connection is \textbf{compatible} with the Finsler metric if the induced parallel transports preserve the Finslerian length of tangent vectors. Finsler manifolds admitting compatible linear connections are called  \textbf{generalized Berwald manifolds}.
\end{defi}

In terms of local coordinates, equations
\begin{equation}
\label{ceq_christoffel}
X_i^h F := \partial_i F-y^j \left( {\Gamma}^k_{ij}\circ \pi \right) \dot{\partial}_k F=0 \hspace{1cm} (i=1,\dots,n)
\end{equation} 
form necessary and sufficient conditions for a linear connection $\nabla$ to be compatible with the Finsler metric $F$. Equations (\ref{ceq_christoffel}) are called \textbf{compatibility equations} or CEQ for short. The fundamental result of generalized Berwald manifold theory  states that a compatible linear connection $\nabla$ is always Riemann metrizable \cite{V5}, i.e. $\nabla$ must be a metric linear connection with respect to a Riemannian metric $\gamma$. Such a Riemannian metric can be given by integration of $g_{ij}$ on the indicatrix hypersurfaces \cite{V5}, see also \cite{Cram} and \cite{Mat1}. It is the so-called \textbf{averaged Riemannian metric}. Therefore CEQ can be reformulated by replacing the Christoffel symbols  $\Gamma_{ij}^k$ by the torsion tensor components \cite{V14}, see also  \cite{RandersGBM}. Using the horizontal vector fields 
\[ X_i^{h*}:=\partial_i -y^j \left( {\Gamma}^{k*}_{ij}\circ \pi \right)  \dot{\partial}_k \hspace{1cm} (i=1,\dots,n), \]
where ${\Gamma}^{k*}_{ij}$ are the Christoffel symbols of the L\'{e}vi-Civita connection of the averaged Riemannian metric $\gamma$, CEQ takes the form
\begin{equation}
y^j \left(T^{l}_{jk}\gamma^{kr}\gamma_{il}+T^{l}_{ik}\gamma^{kr}\gamma_{jl}-T_{ij}^r\right) \dot{\partial}_r F=-2X_i^{h^*}F\hspace{1cm} (i=1,\dots,n).
\end{equation}
The unknown quantities  $T_{ab}^c$ are the torsion tensor components  of the compatible linear connection. 

In what follows we are going to use normal coordinates with respect to the Riemannian metric $\gamma$ around a given point $p\in M$. The coordinate vector fields $\partial/\partial u^1, \dots, \partial/\partial u^n$ form an orthonormal basis in $T_p M$, i.e. $\gamma_{ij}(p)=\delta_{ij}$ and ${\Gamma}^{k*}_{ij}(p)=0$. Therefore $X_i^{h*}(v)=\partial_i(v)$ for any $v\in T_pM$ and CEQ takes the form 
\begin{equation}
\label{CEQ-tors2}
\sum_{a<b,c} \sigma_{ab;i}^c T_{ab}^{c} = -2\partial_i F \hspace{1cm} (i=1, \dots, n),
\end{equation}
where the coefficients are 
\begin{equation}
\label{CEQ-coeff2}
\sigma_{ab;i}^c := \delta_i^a f_{cb} + \delta_i^b f_{ac} + \delta_i^c f_{ab}, \hspace{1cm} f_{ij} := y^i \dot{\partial}_j F - y^j \dot{\partial}_i F.
\end{equation}
If none of the indices $a,b,c$ are equal to $i$ then $\sigma_{ab;i}^c=0$. Otherwise the table shows the possible cases, where indices are separated according to their values (equal indices are put into the same cell and different cells contain different values).

\begin{center}
{\tabulinesep=1pt \begin{tabu} {|c||c|c|c||l|}
\hline 
 & \multicolumn{3}{c||}{\textrm{indices}} & \textrm{the coefficients} \\ 
\hline 
\hline
1. & $i=a$ & $b$ & $c$ & $\sigma_{ib;i}^{c}=f_{cb}$ \\ 
\hline 
2. & $i=a$ & $b=c$ &   & $\sigma_{ib;i}^{b}=0$ \\ 
\hline 
3. & $i=b$ & $a$ & $c$ & $\sigma_{ai;i}^{c}=f_{ac}$ \\ 
\hline 
4. & $i=b$ & $a=c$ &   & $\sigma_{ai;i}^{a}=0$ \\ 
\hline 
5. & $i=c$ & $a$ & $b$ & $\sigma_{ab;i}^{i}=f_{ab}$ \\ 
\hline 
6. & $i=a=c$ & $b$ &   & $\sigma_{ib;i}^{i}=2f_{ib}$ \\ 
\hline 
7. & $i=b=c$ & $a$ &   & $\sigma_{ai;i}^{i}=2f_{ai}$ \\ 
\hline 
\end{tabu} }
\end{center}
Therefore the $i$-th compatibility equation at the point $p$ is
\begin{equation} \label{CEQ-general}
\sideset{}{'}\sum_{a} 2f_{ia} T_{ia}^i + \sideset{}{'}\sum_{a<b} f_{ab} \left( T_{ib}^a + T_{ab}^i + T_{ai}^b \right)  = -2\, \partial_i F,
\end{equation}
where the primed summation means summing for $a \neq i$ in the first one, and in the second one, $a$ and $b$ where $i\neq a, b$. 
  
\section{The geometry of the tangent spaces}
The following table shows a panoramic view about the geometric structures of the tangent space $T_pM$ due to the simultaneously existing Finsler and Riemannian metrics.
\begin{center}
\begin{tabularx}{\textwidth}{|Y|Y|Y|}
\hline 
\textbf{Finsler structure} &   & \textbf{Riemannian structure} \\ 
\hline 
Minkowski norm & metric on $T_p M$ & Euclidean norm and inner product	 \\ 
\hline 
Finslerian spheres $F_p(\lambda)$ & level sets of $\lambda\in \rr_+$ & Euclidean spheres $R_p(\lambda)$ \\ 
\hline 
$\mathcal{F}_v$ & tangent hyperplanes of level sets at $v\in T_pM$ & $\mathcal{R}_v$ \\ 
\hline
$\mathcal{LF}_v:= \mathcal{F}_v - v$  & linear tangent hyperplanes at $v\in T_pM$ & $\mathcal{LR}_v := \mathcal{R}_v - v$ \\ 
\hline
$G := \mathrm{grad} \, F = [ \dot{\partial}_1 F, \dots,  \dot{\partial}_n F]$ & normal vector fields (w.r.t. $\gamma$) & $C := [y^1, \dots, y^n]$ \\ 
\hline 
\end{tabularx} 
\end{center}
Both gradient vector fields $G$ and $C$ are nonzero everywhere on $T_p^{\circ} M := T_p M \backslash \{ \textbf{0} \}$. For every element $v \in T_p^{\circ} M$ the tangent hyperplanes of the Finslerian and the Riemannian (Euclidean) spheres passing through $v$ can be related as follows.
\begin{itemize}
\item If $\mathcal{F}_v=\mathcal{R}_v$, i.e. $G_v \parallel C_v$, we call the point $v$ a \textbf{vertical contact point} of the metrics.
\item If $\mathcal{F}_v \neq \mathcal{R}_v$, i.e. $G_v$ and $C_v$ are linearly independent, then the intersection $\mathcal{LF}_v \cap \mathcal{LR}_v$ is the orthogonal complement of $\mathrm{span}(C_v, G_v)$, and thus $\mathcal{F}_v \cap \mathcal{R}_v$ is an affine subspace of dimension $n-2$.
\end{itemize}

Let us define the vector field
\begin{equation} \label{H-def}
H_v := \left[ \partial_1 F(v), \dots, \partial_n F(v) \right]. 
\end{equation}
An element $v \in T_p^{\circ} M$ is a \textbf{horizontal contact point} of the metrics if $H_v$ is the zero vector. It can be easily seen  that if $T_p M$ is a vertical contact tangent space, i. e. all of its non-zero elements are vertical contact, then the Finsler metric is a scalar multiple of $\gamma$ at the point $p\in M$. The quadratic indicatrix of a generalized Berwald metric at a single point means quadratic indicatrices at all points of the (connected) base manifold because the tangent spaces are related by linear isometries due to the parallel transports with respect to the compatible linear connection with the Finsler metric. Therefore such a generalized Berwald manifold reduces to a Riemannian manifold.  At a horizontal contact point $v\in T_pM$, equations of CEQ are homogeneous. If $T_p M$ is a horizontal contact tangent space, i. e. all of its non-zero elements are horizontal contact, then $T \equiv 0$ is a solution of CEQ at $p\in M$.

\subsection{A useful family of vector fields}  Let us define the vector fields
\begin{equation} \label{rowvector}
f_i(v) := \left[ f_{i1}(v), f_{i2}(v), \dots, f_{in}(v) \right]^T \hspace{1cm} (i=1,\dots, n)
\end{equation}
on $T_p^{\circ} M$ to help in proving some elementary properties and solving CEQ.

\begin{lemm} \label{rowvectorspan}  For any $v \in T_p^{\circ} M$, we have $\mathrm{span}(f_1(v), \dots, f_n(v)) \subseteq  \mathrm{span}(G_v, C_v)$.
\end{lemm}

\begin{proof} Observe that $f_i$ can be written as
\begin{equation} \label{rowvector-lincomb}
f_i = \begin{bmatrix} f_{i1} \\ f_{i2} \\ \vdots \\ f_{in}\end{bmatrix} = 
\begin{bmatrix} y^i \dot{\partial}_1 F - y^1  \dot{\partial}_i F \\ y^i  \dot{\partial}_2 F - y^2  \dot{\partial}_i F \\ \vdots \\ y^i  \dot{\partial}_n F - y^n  \dot{\partial}_i F \end{bmatrix} =
y^i \begin{bmatrix}  \dot{\partial}_1 F \\  \dot{\partial}_2 F \\ \vdots \\  \dot{\partial}_n F \end{bmatrix} -  \dot{\partial}_i F \begin{bmatrix}
y^1 \\ y^2 \\ \vdots \\ y^n \end{bmatrix},
\end{equation}
i.e. $f_i = y^i \cdot G -  \dot{\partial}_i F \cdot C$.
\end{proof}

\begin{lemm} \label{rowvector-vertcont} At a vertical contact point $v \in T_p^{\circ} M$,  $f_i(v)=0$ and  CEQ takes the form $0=\partial_i F(v)$ $(i=1, \ldots, n)$.
\end{lemm}

\begin{proof} If $v$ is vertically contact, then $G_v = \lambda C_v$ for some nonzero $\lambda \in \rr$. Substituting into \eqref{rowvector-lincomb}, we get
\[ f_i(v) = v^i \cdot \lambda C_v - \lambda v^i \cdot C_v = 0 \hspace{1cm} (i=1, \dots, n), \]
i.e. the coordinates of $f_i$ are all zero at $v$ and the reformulation \eqref{CEQ-general} of CEQ implies the statement. 
\end{proof}

\begin{cor} In order for CEQ to have a solution, all vertically contact points must be horizontal contact.
\end{cor}

\begin{lemm} \label{rowvector-notvertcont} At a not vertically contact $v \in T_p^{\circ} M$ there are indices such that $f_{ij}(v)\neq 0$ and the vectors $f_i$ and $f_j$ are linearly independent over some neighborhood $U$ of $v$ in $T_pM$. 
\end{lemm}

\begin{proof} Suppose that all the $f_{ij}(v)$, and consequently, all the vectors $f_i(v)$ are zero. Since $v$ is not vertically contact, $G_v$ and $C_v$ are linearly independent. In particular, neither of them is the zero vector, so one of their coordinates is nonzero, meaning that for some index $i$, \eqref{rowvector-lincomb} gives the zero vector as a linear combination of the independent vectors $C_v$ and $G_v$ with nonzero coefficients. This is a contradiction, so there must be an $f_{ij}(v)$, and thus two vectors $f_i(v)$ and $f_j(v)$ different from zero. Since the matrix
\[ \begin{bmatrix} f_i \\ f_j \end{bmatrix} =
\begin{bmatrix} f_{i1} & \dots & 0 & \dots & f_{ij} & \dots & f_{in} \\ f_{j1} & \dots & f_{ji} & \dots & 0 & \dots & f_{jn}\\ \end{bmatrix} \]
has rank 2 at $v$ (choose the $i$-th and $j$-th columns), they are linearly independent at $v$ and the same is true at the points of some adequately small neighborhood of $v$ in  $T_pM$ by a continuity argument. 
\end{proof}

\begin{cor} \label{rowvector-ultimate} At a point $v \in T_p^{\circ} M$ the vectors defined by \eqref{rowvector} span the subspace
\[ \mathrm{span}(f_1(v), \dots, f_n(v)) = \left\{\begin{array}{cl}
   \{ \bm{0} \} & \text{if} \ v \ \text{is vertically contact}, \\
   \mathrm{span}(G_v, C_v) &  \text{if} \ v \ \text{is not vertically contact}.
 \end{array}\right. \]
\end{cor}

\begin{proof} For any $v \in T_p^{\circ} M$, we have $\mathrm{span}(f_1(v), \dots, f_n(v)) \subseteq \mathrm{span}(G_v, C_v)$ by Lemma \ref{rowvectorspan}. If $v$ is vertically contact, all the vectors $f_i(v)$ are zero according to Lemma \ref{rowvector-vertcont}. If not, there are 2 independent vectors among them according to Lemma \ref{rowvector-notvertcont}, thus generating the whole $\mathrm{span}(G_v, C_v)$.
\end{proof}

\begin{lemm} \label{rowvector-basis} At a not vertically contact $v \in T_p^{\circ} M$, let us choose indices $i \neq j$ such that $f_{ij}(v)\neq 0$. Then $(f_i, f_j)$ is a basis of $\mathrm{span}(G, C)$  over some neighborhood $U$ of $v$ in $T_pM$ and
\begin{equation}\label{rowvector-lincomb2}
f_k = \dfrac{f_{kj}}{f_{ij}} \cdot f_i + \dfrac{f_{ik}}{f_{ij}} \cdot  f_j \hspace{1cm} (k=1, \dots, n)
\end{equation}
at any point of $U$.
\end{lemm}

\begin{proof} By Lemma \ref{rowvector-notvertcont}, we know that $(f_i, f_j)$ is a basis of $\mathrm{span}(G, C)$ at the points of $U$. Let us choose an index $k \in \{ 1, \dots, n\}$ and write $f_k = \lambda_1 f_i + \lambda_2 f_j$. By \eqref{rowvector-lincomb}, we can write that
\[f_k= y^k \cdot G - \dot{\partial}_k F \cdot C =  \lambda_1 \left( y^i \cdot G - \dot{\partial}_i F \cdot C \right) + \lambda_2 \left( y^j \cdot G - \dot{\partial}_j F \cdot C \right). \]
By comparing the coefficients in the basis $(G, C)$,
\[ \begin{bmatrix} y^k \\ \dot{\partial}_k F \end{bmatrix} = \begin{bmatrix} y^i & y^j \\ \dot{\partial}_i F & \dot{\partial}_j F \end{bmatrix} \cdot \begin{bmatrix} \lambda_1 \\ \lambda_2
\end{bmatrix} \]
and, consequently, 
\[ \begin{bmatrix} \lambda_1 \\ \lambda_2
\end{bmatrix} = \dfrac{1}{f_{ij}} \begin{bmatrix} \dot{\partial}_j F & -y^j \\ -\dot{\partial}_i F & y^i \end{bmatrix} \cdot \begin{bmatrix} y^k \\ \dot{\partial}_k F \end{bmatrix} = \dfrac{1}{f_{ij}} \begin{bmatrix} y^k \dot{\partial}_j F - y^j \dot{\partial}_k F \\ y^i \dot{\partial}_k F - y^k \dot{\partial}_i F \end{bmatrix}. \qedhere \]
\end{proof}

\section{The semi-symmetric compatible linear connection and its unicity}

Although there are some theoretical results for the solvability of the compatibility equations  (monochromatic Finsler metrics \cite{BM}, extremal compatible linear connections and algorithmic solutions \cite{V14}), it is very hard to solve in general because compatible linear connections may or may not exist on a Finsler manifold and may or may not be unique.  Therefore special cases are of special interest. One of them is the case of the so-called semi-symmetric compatible linear connection with decomposable torsion tensor. 

\begin{defi} A linear connection is called \textbf{semi-symmetric} if its torsion tensor can be written as
\begin{equation}\label{semi-symm-tors}
T(X,Y) = \rho(Y) X - \rho(X)Y
\end{equation}
for some differential 1-form $\rho$ on the base manifold.
\end{defi}

It is proved \cite{V10} that a semi-symmetric compatible linear connection must be uniquely determined.  

\begin{thm}\label{original} \emph{\cite{V10}} A non-Riemannian Finsler manifold admits at most one semi-symmetric compatible linear connection.
\end{thm}

The original  proof is based on averaging in the sense that the 1-form $\rho$  can be expressed by integrating differential forms on the tangent manifold over the Finslerian indicatrices. The integral formulas are very difficult to compute in practice. In what follows we present a new proof for the unicity result by using linear algebra and some basic facts about convex bodies. We present an explicit formula for the solution without integration. The method has a new contribution to the problem as well:  necessary conditions of the solvability are formulated in terms of intrinsic equations without unknown quantities. They are sufficient if and only if the solution depends only on the position.  

\subsection{The proof of Theorem \ref{original}} Since 
$$T(\partial/\partial u^i, \partial/\partial u^j) =  \rho(\partial/\partial u^j) \partial/\partial u^i - \rho(\partial/\partial u^i) \partial/\partial u^j=$$
$$= \rho_j \partial/\partial u^i - \rho_i \partial/\partial u^j = \left( \delta_i^k \rho_j - \delta_j^k \rho_i \right) \partial/\partial u^k,$$
the torsion components are
\begin{equation} \label{semi-symm-torscomp}
T_{ij}^k = \delta_i^k \rho_j - \delta_j^k \rho_i.
\end{equation}
In particular, all torsion components with 3 different indices are zero, and 
\[ T_{ij}^i = \delta_i^i \rho_j - \delta_j^i \rho_i =\rho_j \hspace{1cm }(j \neq i). \]
Substituting the torsion components into the general form \eqref{CEQ-general} of CEQ at a point $p$, it takes the (matrix) form
\begin{equation} \label{CEQ}
{\tabulinesep=2pt
\begin{tabu} to .5\textwidth {|ccccc|c|}
\hline 
\rho_1 & \rho_2 & \rho_3 & \cdots & \rho_n & \text{RHS} \\ 
\hline 
0 & f_{12} & f_{13} & \cdots & f_{1n} & -\partial_1 F \\ 
f_{21} & 0 & f_{23} & \cdots & f_{2n} & -\partial_2 F \\ 
f_{31} & f_{32} & 0 & \cdots & f_{3n} & -\partial_3 F \\
\vdots & \vdots & \vdots & \ddots & \vdots & \vdots \\ 
f_{n1} & f_{n2} & f_{n3} & \cdots & 0 & -\partial_n F \\
\hline 
\end{tabu}} 
\end{equation}
The problem is to solve \eqref{CEQ} for $\rho_1, \dots, \rho_n$, considered as the coordinates of a vector $\rho \in T_p M$, as $v$ ranges over $T_p^{\circ} M$. To prove Theorem \ref{original} it is enough to consider the homogeneous version  H-CEQ with vanishing right hand side of  \eqref{CEQ}. We are going to verify that the only solution of H-CEQ is $\rho_1=\dots=\rho_n=0$. Since the rows of the matrix on the left hand side are exactly the vectors $f_1, \dots, f_n$ defined in \eqref{rowvector}, solving H-CEQ at a fixed element $v$ means finding the orthogonal complement of $\mathrm{span}(f_1(v), \dots, f_n(v))$. By Corollary \ref{rowvector-ultimate},
\begin{itemize}
\item it is $T_p M$ for any vertically contact element $v$,
\item  it is the orthogonal complement of $\mathrm{span}(G_v, C_v)$ for any not vertically contact element $v$, i.e. the intersection $\mathcal{LF}_v \cap \mathcal{LR}_v$ of the linear tangent hyperplanes of the Finslerian and Riemannian spheres. 
\end{itemize}
The solution of H-CEQ at the point $p$ is the intersection of all the solution spaces as the element $v$ ranges over $T_p^{\circ} M$. Note that the homogeneity of the coefficients imply that it is enough to consider the intersection of all the solution spaces as the element $v$ ranges over the Finslerian (or the Riemannian) unit sphere. 
\begin{itemize}
\item If all the elements of $T_p M$ are vertically contact and the Finsler manifold admits a compatible (semi-symmetric) linear connection $\nabla$, then it is a Riemannian manifold because the linear isometries via the parallel transports with respect to $\nabla$ extend the quadratic Finslerian (esp. Riemannian) indicatrix at the point $p$ to the entire (connected) manifold. 
\item If there is a not vertically contact element $v$, then, by a continuity argument, we can consider a neighborhood $U\subseteq T_pM$ containing not vertically contact elements. For the solution vector $\rho$ we have
\[ \rho \in \bigcap_{v \in U} \left( \mathcal{LF}_v \cap \mathcal{LR}_v \right) \subseteq \big(\bigcap_{v \in U} \mathcal{LF}_v \big) \cap \big(\bigcap_{v \in U} \mathcal{LR}_v \big). \]
It is clear that  the right hand side contains only the zero vector because the normal vectors at the points of any open set on the boundary of a Euclidean sphere (or any smooth strictly convex body) span the entire space. Therefore $\rho=0$ is the only solution of H-CEQ at $p$ and, consequently, CEQ admits at most one solution for the components of the torsion tensor of a semi-symmetric linear connection point by point. \qedhere
\end{itemize}

\subsection{Intrinsic equations and $v$-solvability of CEQ} In this section we investigate \eqref{CEQ} evaluated at non-zero tangent vectors in $T_pM$:

\begin{equation} \label{v-CEQ}
{\tabulinesep=2pt
\begin{tabu} to .5\textwidth {|ccccc|c|}
\hline 
\rho_1 & \rho_2 & \rho_3 & \cdots & \rho_n & \text{RHS} \\ 
\hline 
0 & f_{12}(v) & f_{13}(v) & \cdots & f_{1n}(v) & -\partial_1 F(v) \\ 
f_{21}(v) & 0 & f_{23}(v) & \cdots & f_{2n}(v) & -\partial_2 F(v) \\ 
f_{31} (v)& f_{32}(v) & 0 & \cdots & f_{3n}(v) & -\partial_3 F(v) \\
\vdots & \vdots & \vdots & \ddots & \vdots & \vdots \\ 
f_{n1} (v)& f_{n2}(v) & f_{n3} (v)& \cdots & 0 & -\partial_n F(v) \\
\hline 
\end{tabu}} 
\end{equation} 

\begin{defi} The system of the compatibility equations is called \textbf{$v$-solvable} for $\rho$ at the point $p\in M$ if \eqref{v-CEQ} is solvable for any non-zero element $v\in T_pM$.
\end{defi}

\begin{rem}
The system of the compatibility equations is $v$-solvable if and only if all vertical contact vectors are horizontal contact. It is an obvious necessary condition because the coefficient matrix of CEQ is zero at a vertically contact point and the system must be homogeneous with vanishing horizontal derivatives of the Finsler metric with respect to the compatible linear connection. The sufficiency is based on the idea of the extremal compatible linear connection \cite{V14}. The extremal solution typically depends on the reference element $v\in T_pM$  but does not take a decomposable form in general. Therefore $v$-solvability for $\rho$ needs additional conditions. Using Corollary \ref{rowvector-ultimate} and basic linear algebra we can formulate the following characterizations of $v$-solvability in case of semi-symmetric compatible linear connections. 
\end{rem}

\begin{lemm} \label{v-solvability1}  The system of the compatibility equations is $v$-solvable for $\rho$ at the point $p\in M$ if and only if the following are satisfied:
\begin{itemize}
\item all vertically contact elements are also horizontal contact in $T_pM$ and
\item the rank of the augmented matrix of system \eqref{v-CEQ} is 2  at all not vertically contact elements $v\in T_pM$, i.e. $H_v \in \mathrm{span}(f_1(v), \dots, f_n(v))$, where the vector $H_v$ is defined by formula \eqref{H-def}.
\end{itemize}
\end{lemm}

\begin{prop} \label{v-solvability2} The system of the compatibility equations is $v$-solvable for $\rho$ at the point $p\in M$ if and only if the following are satisfied:
\begin{itemize}
\item all vertically contact elements are also horizontal contact  in $T_pM$  and
\item for any triplets of distinct indices $i,j,k$ we have
\begin{equation} \label{solv-symm}
f_{ij}(v) \, \partial_k F (v)+ f_{jk}(v) \, \partial_i F (v)+ f_{ki}(v) \, \partial_j F(v) = 0
\end{equation}
provided that  $f_{ij}(v)\neq 0$.
\end{itemize}
\end{prop}

\begin{proof} If $v$ is vertically contact, \eqref{solv-symm} stands trivially. Otherwise we are going to show that it is equivalent to the augmented matrix having rank 2. Suppose that $v$ is not vertically contact and choose indices $i\neq j$ such that $f_{ij}(v)\neq 0$ and $(f_i(v), f_j(v))$ is a basis of $\mathrm{span}(G_v, C_v)$. By Lemma \ref{rowvector-basis}, if $k\neq i, j$ then 
\[ f_k = \dfrac{f_{kj}}{f_{ij}} \cdot f_i + \dfrac{f_{ik}}{f_{ij}} \cdot  f_j \hspace{1cm} (k=1, \dots, n). \]
In other words we can eliminate the $k$-th row for any $k\neq i, j$. The elimination must also yield zeroes on the right-hand side of \eqref{v-CEQ} to have a solution, i.e. for $k \neq i, j$ we must have
\[ \begin{array}{c}
-\partial_k F - \dfrac{f_{kj}}{f_{ij}} \cdot (-\partial_i F) - \dfrac{f_{ik}}{f_{ij}} \cdot (-\partial_j F) = 0 \\[12pt]
f_{ij} \, \partial_k F - f_{kj} \, \partial_i F - f_{ik} \, \partial_j F = 0.
\end{array} \]
Equation \eqref{solv-symm} follows by interchanging the indices in $f_{kj}$ and $f_{ik}$. Since the coefficient matrix is of maximal rank $2$, the augmented matrix is of maximal rank $2$ as the extension of the coefficient matrix with $-\partial_i F$ and $-\partial_j F$ in the corresponding rows.   
\end{proof}

\begin{rem} Equations \eqref{solv-symm} do not contain unknown quantities. They are intrinsic conditions of the solvability. Taking $f_{ij}(v)\neq 0$ for some fixed indices $i$ and $j$ at a not vertically contact element, they provide $n-2$ equations to be automatically satisfied because $k=1, \ldots, n$, but $k\neq i, j$. The missing equations are  
\begin{equation}  \label{CEQ-elim}
\left.\begin{array}{rcl}
 \dotprod{f_i,\rho} & = & -\partial_i F \\[4pt]
 \dotprod{f_j,\rho} & = & -\partial_j F \\
 \end{array}\right\},
\end{equation}
where  $\dotprod{f_i,\rho}$ and $ \dotprod{f_j,\rho}$ stand for the inner product at $p\in M$ coming from the Riemannian metric $\gamma$. They provide the only possible solution $\rho$  in an explicit form. 
\end{rem}

\subsection{The only possible solution of CEQ at the point $p\in M$} \label{lastsect} Recall that if the tangent space at $p\in M$ contains only vertically contact non-zero elements (vertically contact tangent space)  and the Finsler manifold admits a compatible (semi-symmetric) linear connection $\nabla$, then it is a Riemannian manifold because the linear isometries via the parallel transports with respect to $\nabla$ extend the quadratic Finslerian (esp. Riemannian) indicatrix at the point $p$ to the entire (connected) manifold. Therefore we present the solution of CEQ in the generic case of non-Riemannian Finsler manifolds. Let us choose a not vertically contact element $v \in T_p^{\circ} M$ and indices $i\neq j$ with $f_{ij}(v)\neq 0$, i.e. $f_i, f_j$ are linearly independent over some neighborhood $U$ of $v$ in $T_pM$. Using \eqref{rowvector-lincomb}, the eliminated form \eqref{CEQ-elim} of CEQ gives that 
\[  {\arraycolsep=1pt \left.\begin{array}{ccccl}
 y^i \dotprod{G,\rho} & - & \dot{\partial}_i F \dotprod{C,\rho} & = & -\partial_i F \\[5pt]
 y^j \dotprod{G,\rho} & - & \dot{\partial}_j F \dotprod{C,\rho} & = & -\partial_j F \\
 \end{array}\right\} } \Longleftrightarrow
\begin{bmatrix} y^i & -\dot{\partial}_i F \\[5pt] y^j & -\dot{\partial}_j F  \end{bmatrix} \cdot \begin{bmatrix} \dotprod{G,\rho} \\[5pt] \dotprod{C,\rho} \end{bmatrix} = \begin{bmatrix}
-\partial_i F \\[5pt] -\partial_j F
\end{bmatrix}, \]
and, consequently,
\[ \begin{bmatrix} \dotprod{G,\rho} \\[5pt] \dotprod{C,\rho} \end{bmatrix} = \dfrac{1}{f_{ji}}\begin{bmatrix}
-\dot{\partial}_j F &\dot{\partial}_i F \\[5pt] -y^j & y^i
\end{bmatrix} \cdot \begin{bmatrix}
-\partial_i F \\[5pt] -\partial_j F
\end{bmatrix}. \]
We are going to concentrate on the second row
\begin{equation} \label{eq-c}
\dotprod{C,\rho} = \dfrac{1}{f_{ji}} \left( y^j \partial_i F - y^i \partial_j F \right) =: \dfrac{f_{ji}^h}{f_{ji}}
\end{equation} 
at the points of the open neighborhood $U$ around $v$. Let us choose a value $\varepsilon>0$ such that all the elements 
\[ \begin{array}{rcccl}
w_1 & := & v-\varepsilon \cdot \partial/\partial u^1(p) & = & [v^1-\varepsilon, v^2, v^3, \dots, v^n] \\
w_2 & := & v-\varepsilon \cdot \partial/\partial u^2(p) &= & [v^1, v^2-\varepsilon, v^3, \dots, v^n] \\
& \vdots &&& \\
w_n & := & v-\varepsilon \cdot \partial/\partial u^n(p) & = & [v^1, v^2, v^3, \dots, v^n-\varepsilon]
\end{array} \]
are contained in $U$. Then \eqref{eq-c} implies the system  
\begin{equation} \label{ceq-indep} \begin{bmatrix}
v^1-\varepsilon & v^2 & v^3 & \cdots & v^n \\[4pt]
v^1 & v^2-\varepsilon & v^3 & \dots & v^n \\[4pt]
\vdots & \vdots & \vdots & \ddots & \vdots \\[4pt]
v^1 & v^2 & v^3 & \dots & v^n-\varepsilon
\end{bmatrix} \cdot \begin{bmatrix} \rho_1 \\[4pt] \rho_2 \\[4pt] \vdots \\[4pt] \rho_n \end{bmatrix} = \begin{bmatrix}
f^h_{ji}/f_{ji}(w_1) \\[4pt] f^h_{ji}/f_{ji}(w_2) \\[4pt] \vdots \\[4pt] f^h_{ji}/f_{ji}(w_n) \end{bmatrix}
\end{equation}
of linear equations. Using the notation
\[ V:= \begin{bmatrix}
v^1 & v^2 & \cdots & v^n \\
\vdots & \vdots & \ddots & \vdots \\
v^1 & v^2 & \cdots & v^n
\end{bmatrix}, \]
we have to investigate the regularity of the matrix $V-\varepsilon I$, where $I$ denotes the identity matrix of the same type as $V$.

\begin{lemm} The matrix $V-\varepsilon I$ is regular if and only if $\varepsilon \notin \{ 0, \widetilde{v}:=v^1+\dots+v^n \}$.
\end{lemm}

\begin{proof} The determinant $\mathrm{det}(V-\varepsilon I)$ is the characteristic polynomial of $V$ with $\varepsilon$ as the variable. It is zero if and only if $\varepsilon$ is an eigenvalue of $V$. Consider the transpose of $V$ as the matrix of a linear transformation $\varphi$ (the eigenvalues are the same as those of $V$). Since the image of $\varphi$ is the line generated by $v$, it follows that $\varphi$ has a kernel of dimension $n-1$ and $v$ is an eigenvector corresponding to the eigenvalue $\widetilde{v}:=v^1+\dots+v^n$ because of
\[ \varphi(v)= \begin{bmatrix}
v^1 & \cdots & v^1 \\ \vdots & \ddots & \vdots \\ v^n & \cdots & v^n \end{bmatrix} \cdot \begin{bmatrix} v^1 \\ \vdots \\ v^n\end{bmatrix} = \begin{bmatrix}
v^1 (v^1+\dots+v^n) \\ \vdots \\v^n(v^1+\dots+v^n)
\end{bmatrix} = \widetilde{v} \cdot v. \qedhere\]
\end{proof}

\begin{lemm} Choosing $\varepsilon \notin \{ 0, \widetilde{v}:=v^1+\dots+v^n\}$, we have
\[ \left[ V-\varepsilon I \right]^{-1} = \dfrac{1}{(\widetilde{v}-\varepsilon)\varepsilon} \left[ V-(\widetilde{v}-\varepsilon) I \right]. \]
\end{lemm}

\begin{proof} We shall prove the formula
\[ \left[V-\varepsilon I\right] \cdot \left[ V-(\widetilde{v}-\varepsilon)I\right] = (\widetilde{v}-\varepsilon)\varepsilon I.\]
Rearranging the left-hand side, 
\[ V \cdot V - (\widetilde{v}-\varepsilon)V-\varepsilon V + \varepsilon (\widetilde{v}-\varepsilon)I=(\widetilde{v}-\varepsilon)\varepsilon I \]
because of $V\cdot V=\widetilde{v} \cdot V$. 
\end{proof}

Returning to \eqref{ceq-indep}, 
\[\begin{bmatrix} \rho_1 \\[4pt] \rho_2 \\[4pt] \vdots \\[4pt] \rho_n \end{bmatrix} =  \dfrac{1}{(\widetilde{v}-\varepsilon)\varepsilon} \begin{bmatrix}
v^1+\varepsilon-\widetilde{v} & v^2 &  \cdots & v^n \\[4pt]
v^1 & v^2+\varepsilon-\widetilde{v} &  \dots & v^n \\[4pt]
\vdots & \vdots &  \ddots & \vdots \\[4pt]
v^1 & v^2 &  \dots & v^n+\varepsilon-\widetilde{v}
\end{bmatrix} \cdot \begin{bmatrix}
f^h_{ji}/f_{ji}(w_1) \\[4pt] f^h_{ji}/f_{ji}(w_2) \\[4pt] \vdots \\[4pt] f^h_{ji}/f_{ji}(w_n) \end{bmatrix}. \]
 \begin{thm} If a non-Riemannian Finsler manifold admits a semi-symmetric compatible linear connection, then the only possible values of the components $\rho_k$ in formula  \eqref{semi-symm-torscomp} for its torsion at the point $p\in M$ are
\begin{equation}
\rho_k = \dfrac{1}{\varepsilon} \left( \dfrac{1}{\widetilde{v}-\varepsilon} \sum_{l=1}^n v^l \dfrac{f^h_{ji}}{f_{ji}}(w_l) - \dfrac{f^h_{ji}}{f_{ji}}(w_k) \right),
\end{equation}
where,
\begin{itemize}
\item $v=[v^1, \dots, v^n]$ is a not vertically contact vector in $T^{\circ}_p M$,
\item $\varepsilon \notin \{ 0, \widetilde{v}:=v^1+\dots+v^n\}$ is a radius of a closed ball around $v$ in $T_pM$ all of whose elements are also not vertically contact,
\end{itemize}
\begin{itemize}[itemsep=6pt]
\item $w_i=[v^1, \dots, v^{i-1}, v^i-\varepsilon, v^{i+1}, \dots, v^n]$, $i=1, \ldots, n$,
\item $f^h_{ji}= y^j \partial_i F - y^i \partial_j F = y^j \dfrac{\partial F}{\partial x^i} - y^i \dfrac{\partial F}{\partial x^j}$,
\item $f_{ji}= y^j \dot{\partial}_i F - y^i \dot{\partial}_j F = y^j \dfrac{\partial F}{\partial y^i} - y^i \dfrac{\partial F}{\partial y^j}$ and 
\item the coordinates on the base manifold form a normal coordinate system with respect to the averaged Riemannian metric $\gamma$ around the point $p\in M$.
\end{itemize}
\end{thm}

\section{Acknowledgments}

Márk Oláh is supported by the UNKP-21-3 New National Excellence Program of the Ministry for Innovation and Technology from the source of the National Research, Development and Innovation Fund of Hungary.

\end{document}